\documentclass[12pt, reqno]{amsart}
\usepackage[a4paper, total={6in, 8in}]{geometry}
\usepackage{graphicx}
\usepackage{amssymb, amsthm}
\usepackage{epstopdf}
\usepackage{xypic}

\usepackage[enableskew]{youngtab}

\usepackage{caption}
\usepackage{subcaption}
\usepackage{ytableau}

\theoremstyle{plain}
\newtheorem{theorem}{Theorem}[section]

\theoremstyle{definition}
\newtheorem{remark}[theorem]{Remark}
\newtheorem{lemma}[theorem]{Lemma}

\def\L{\mathcal{L}}
\def\R{\mathbb{R}}

\def\e{\mathbb{E} }
\def\p{\mathbb{P} }

\def\X{\mathbb{X}}
\def\Y{\mathbb{Y}}
\newcommand{\ignore}[1]{ }

\def\a{\alpha}
\def\b{\beta}
\def\RR{{\rm RR}}
\def\CC{{\rm CC}}
\def\CR{{\rm CR}}
\def\RC{{\rm RC}}
\def\SS{{\rm SS}}

\def\IRR{I_{\RR}}
\def\ICC{I_{\CC}}
\def\ICR{I_{\CR}}
\def\IRC{I_{\RC}}
\def\ISS{I_{\SS}}

\def\WR{W_{\RR}}
\def\WC{W_{\CC}}
\def\WS{W_{\SS}}

\def\bfD{{\bf D}}
\def\bfC{{\bf C}}

\def\.{\hskip.06cm}

\title[Attacks and alignments of rooks]{Attacks and alignments:  rooks,  set partitions, and permutations} 

\author{Richard Arratia}
\address{University of Southern California \\ Los Angeles, CA, USA}
\email{rarratia@usc.edu}

\author{Stephen DeSalvo}
\address{University of California, Los Angeles \\ Los Angeles, CA, USA}
\email{stephen.desalvo@math.ucla.edu}

\date{June 29, 2021}

\begin{document}
\begin{abstract}
We consider uniformly random set partitions of size $n$ with exactly $k$ blocks, and uniformly random permutations of size~$n$ with exactly~$k$ cycles, under the regime where $n-k \sim t\sqrt{n}$, $t>0$.
In this regime, there is a simple approximation for the entire process of component counts;  in particular, the number of components of size 3 converges in distribution to Poisson with mean $\frac{2}{3}t^2$ for set partitions and mean $\frac{4}{3}t^2$ for permutations, and with high probability all other components have size one or two.
These  approximations are proved, with preasymptotic error bounds, using combinatorial bijections for placements of  $r$ rooks on a triangular half of an $n\times n$ chess board, together with the Chen--Stein method for processes of indicator random variables.
\end{abstract}

\maketitle

\section{Introduction}
\label{new_introduction}

We exploit two combinatorial bijections, each involving non-attacking rooks on a lower triangular chess board, to describe the component structure of set partitions and permutations under the regime where the size~$n$ and the number of components~$k$ satisfies $r \equiv r(n,k) := n-k \sim t\, \sqrt{n}$, $t>0$. 
A similar analysis, focused on the approximation of Stirling numbers, and exploiting only the existence of the bijections rather than their explicit forms, was carried out by the authors in~\cite{Stirling} using Chen--Stein Poisson approximation~\cite{chen1975poisson}. 
In this paper we apply the Poisson process approximation approach outlined in~\cite{ArratiaGoldstein} to fully characterize the component structure. 

The total variation distance between the \emph{distributions} of two random variables $X$ and $Y$ in $\R^n$ is defined as
\[ d_{TV}(\L(X), \L(Y)) = \sup_{A \subset \R^n} |\p(X \in A) - \p(Y \in A)|, \]
where the $\sup$ is taken over all Borel sets $A$.  
When there is no confusion, we instead write $d_{TV}(X,Y)$ to denote the total variation distance between the distributions $\L(X)$ and $\L(Y)$. 

Our main results, both of which follow from Lemma~\ref{theorem:main}, are summarized below.
The first result concerns the cycle lengths in a random permutation of size~$n$ into exactly~$k$ cycles.

\begin{theorem}\label{cycle:sizes:extension}
For each integer $i \geq 1$, let $C_i \equiv C_i(n,k)$ denote the number of cycles of size~$i$ in a random permutation of size~$n$ into exactly $k$ cycles, and denote the joint distribution of cycle sizes by $\bfC \equiv \bfC(n,k) := (C_1, C_2, \ldots)$.
Let $W_1 \equiv W_1(n,k)$ denote a Poisson random variable with $\lambda_1 := \e W_1 = \frac{4}{3} \frac{(n-k)^2}{n}$. 
If $k(1), k(2), \ldots$ is an increasing sequence of nonnegative integers such that $n - k(n) \sim t\, \sqrt{n}$ as $n$ tends to infinity, then we have (with $k \equiv k(n)$) 
 \begin{align}
\label{dtv1}
 \begin{split}
 d_{TV}(\ \bfC,\ (n-2(n-k)+W_1, & (n-k)-2W_1, W_1, 0, 0, \ldots)\ )  \\ 
& = O\left(\exp\left(\frac{2}{3}{\frac{(n-k)^2}{n}}\right)\, \cdot \,  \frac{(n-k)^3}{n^2}\right). 
\end{split}
\end{align}
\end{theorem}

The second result is an analogous theorem for the block sizes in a random set partition of size~$n$ into exactly~$k$ blocks.

\begin{theorem}\label{block:sizes:extension}
For each integer $i \geq 1$, let $D_i \equiv D_i(n,k)$ denote the number of blocks of size~$i$ in a  random set partition of size~$n$ into exactly $k$ blocks, and denote the joint distribution of block sizes by $\bfD \equiv \bfD(n,k) := (D_1, D_2, \ldots)$. 
Let $W_2 \equiv W_2(n,k)$ denote a Poisson random variable with $\lambda_2 := \e W_2 = \frac{2}{3}\frac{(n-k)^2}{n}$.
If $k(1), k(2), \ldots$ is an increasing sequence of nonnegative integers such that $n - k(n) \sim t\, \sqrt{n}$ as $n$ tends to infinity, then we have (with $k \equiv k(n)$)
\begin{align}\label{dtv}
 \begin{split}
 d_{TV}(\ \bfD ,\ (n-2(n-k)+W_2, & (n-k)-2W_2, W_2,0,0,\ldots)\ ) \\
  & = O\left(\exp\left(\frac{4}{3}{\frac{(n-k)^2}{n}}\right)\, \cdot \,  \frac{(n-k)^3}{n^2}\right). 
\end{split}
\end{align}
\end{theorem}

In Section~\ref{the:bijection}, we describe the two bijections involving the placements of rooks on a chess board, after which we provide a qualitative explanation of the theorem above. 
In Section~\ref{poisson:approximation}, we define the random variables of interest and outline the Poisson approximation approach.
In Section~\ref{rook_process}, we combine these two techniques to prove preasymptotic bounds for Poisson process approximation, summarized in Lemma~\ref{theorem:main}, which are used to prove the main results above.   
Finally, in Section~\ref{applications} we present some applications and corollaries. 

\section{Approach}
\subsection{Rook bijections}
\label{the:bijection}

Each of these structures, viz., set partitions and permutations, has an intimate relation with placements of $r$ non-challenging rooks on the triangular board 
$$
   B \equiv B_n := \{(i,j):  1 \le i < j \le n \}.
$$   
For the case of set partitions, two rooks  \emph{challenge} or \emph{attack} as per the usual rules of chess if they lie in the same row or column.  For the case of permutations, two rooks are said to \emph{challenge}  or \emph{attack} if they lie in the same column.
We write $S(n,k)$ for the Stirling number of the second kind, counting the number of partitions of a set of size $n$ into a set of exactly $k$ blocks.  We write $s(n,k)$ for the Stirling number of first kind, whose absolute value counts the number of permutations of a set of size $n$ having exactly $k$ cycles.   
In either case, the Stirling number is equal to the number of ways to place $r$ unlabelled rooks on $B_n$ with no attacks.

We utilize two different bijections involving the placements of rooks on the board $B$, one for set partitions and the other for permutations.  What follows is a partial description of those bijections for the purpose of understanding the proofs of our results, and we refer the interested reader to~\cite{rooktheorynotes} for the complete proofs of these bijections.  

\begin{figure}[h]
\begin{subfigure}[t]{0.25\textwidth}
\ytableausetup{centertableaux, boxsize=1.2em}
\begin{ytableau}
\none[{\color{gray} 6}] & \none  &\none  &\none  &\none &\none &\none  \\ 
\none[{\color{gray} 5}] & \none  &\none  &\none  &\none &\none &\ \\ 
\none[{\color{gray} 4}] & \none  &\none  &\none  &\none &\  &\ \\ 
\none[{\color{gray} 3}] & \none  &\none  &\none  &\ &\  &\ \\ 
\none[{\color{gray} 2}] & \none  &\none  &\ & a  &\ &b \\ 
\none[{\color{gray} 1}] & \none  &\  &\ &\  &\ & \ \\
\none & \none[{\color{gray} 1}] & \none[{\color{gray} 2}] & \none[{\color{gray} 3}] & \none[{\color{gray} 4}] & \none[{\color{gray} 5}] & \none[{\color{gray} 6}]
\end{ytableau}
\caption{RR}\label{block3_RR}
\end{subfigure}~
\begin{subfigure}[t]{0.25\textwidth}
\begin{ytableau}
\none[{\color{gray} 6}] & \none  &\none  &\none  &\none &\none &\none \\ 
\none[{\color{gray} 5}] & \none  &\none  &\none  &\none &\none &\ \\ 
\none[{\color{gray} 4}] & \none  &\none  &\none  &\none &\  &a \\ 
\none[{\color{gray} 3}] & \none  &\none  &\none  &\ &\  &\ \\ 
\none[{\color{gray} 2}] & \none  &\none  &\ &\  &\ &b \\ 
\none[{\color{gray} 1}] & \none  &\  &\ &\  &\ & \ \\
\none & \none[{\color{gray} 1}] & \none[{\color{gray} 2}] & \none[{\color{gray} 3}] & \none[{\color{gray} 4}] & \none[{\color{gray} 5}] & \none[{\color{gray} 6}]
\end{ytableau}
\caption{CC}
\end{subfigure}~
\begin{subfigure}[t]{0.25\textwidth}
\begin{ytableau}
\none[{\color{gray} 6}] & \none  &\none  &\none  &\none &\none &\none \\ 
\none[{\color{gray} 5}] & \none  &\none  &\none  &\none &\none &\ \\ 
\none[{\color{gray} 4}] & \none  &\none  &\none  &\none &b  &\ \\ 
\none[{\color{gray} 3}] & \none  &\none  &\none  &\ &\  &\ \\ 
\none[{\color{gray} 2}] & \none  &\none  &\ &a  &\ &\ \\ 
\none[{\color{gray} 1}] & \none  &\  &\ &\  &\ & \ \\
\none & \none[{\color{gray} 1}] & \none[{\color{gray} 2}] & \none[{\color{gray} 3}] & \none[{\color{gray} 4}] & \none[{\color{gray} 5}] & \none[{\color{gray} 6}]
\end{ytableau}
\caption{RC}
\end{subfigure}
~
\begin{subfigure}[t]{0.25\textwidth}
\begin{ytableau}
\none[{\color{gray} 6}] & \none  &\none  &\none  &\none &\none &\none \\ 
\none[{\color{gray} 5}] & \none  &\none  &\none  &\none &\none &\ \\ 
\none[{\color{gray} 4}] & \none  &\none  &\none  &\none &a  &\ \\ 
\none[{\color{gray} 3}] & \none  &\none  &\none  &\ &\  &\ \\ 
\none[{\color{gray} 2}] & \none  &\none  &\ &b &\ &\ \\ 
\none[{\color{gray} 1}] & \none  &\  &\ &\  &\ & \ \\
\none & \none[{\color{gray} 1}] & \none[{\color{gray} 2}] & \none[{\color{gray} 3}] & \none[{\color{gray} 4}] & \none[{\color{gray} 5}] & \none[{\color{gray} 6}]
\end{ytableau}
\caption{CR}
\end{subfigure}
\caption{The board $B_n:= \{(i,j):  1 \le i < j \le n )$, drawn in French notation:  square $(i,j)$ is in the $i$th row from bottom to top, and in the $j$th column from left to right; the case $n=6$ is illustrated.   A rook at $(i,j)$ forces $i$ and $j$ to belong to the same component, implying a block of size at least two.
For two rooks, say $a$ and $b$ with $1 \le a < b \le r$, there are four kinds of possible coincidence, according to the row or column coordinate of rook $a$ being equal to the row or column coordinate of rook $b$,  as shown.
In the case of set partitions, \emph{attacks} are the RR and CC coincidences;  in the case of permutations, \emph{attacks} are the CC coincidences.
In both cases, the non-attack coincidences are \emph{alignments}, implying blocks of size three or more.
}
\label{block 2}
\end{figure}

The first bijection is for set partitions, and requires that no two rooks lie in the same row or column.
A board with no rooks corresponds to the set partition $\{ \{1\}, \{2\}, \ldots, \{n\} \}$, i.e., the set partition with~$n$ blocks of size~$1$.
The placement of a single rook at coordinate $(i,j)$, $1\leq i < j \leq n    $, combines the two separate blocks $\{i\}, \{j\}$, into a single block of size~$2$, i.e., $\{i,j\}$.
The placement of two rooks in coordinates $(i,j)$ and $(j,\ell)$, $1 \leq i < j < \ell \leq n$,
where the row-coordinate of one rook is the column-coordinate of the other, is what we call an \emph{alignment}.  
This corresponds to combining blocks $\{i,j\}$ and $\{j,\ell\}$ into a single block $\{i,j,\ell\}$ of size~$3$; see Figure~\ref{block 3}. 
Further alignments correspond to combining blocks in an analogous manner. 

The second bijection is for permutations, and requires that no two rooks lie in the same column.
A board with no rooks corresponds to the identity permutation $(1)(2)\cdots (n)$. 
The placement of a single rook at coordinate $(i,j)$, $1 \leq i<j \leq n$, creates a cycle of length two, namely, $(i\, j)$, without changing the other fixed points. 
The placement of another rook at coordinate $(i,k)$, $1 \leq i < k \leq n$ and $j \neq k$, i.e., in the same row, creates a cycle of length three, either $(i\ j\ k)$ or $(i\ k\ j)$ depending on whether $j < k$ or $j > k$, respectively, and corresponds to the event RR in Figure~\ref{block3_RR}. 
In general, we have the following rules:
\begin{enumerate}
\item element $i$ is a fixed point if there are no rooks in row $i$ or column $i$, $i \geq 1$;
\item a cycle of length 2, say $(i\ j)$, occurs when there is exactly one rook in row $i$ and column $j$, and no other rooks in row $j$;
\item a cycle of length 3 or more occurs when two or more rooks lie in the same row, or two or more rooks are in alignment.  
\end{enumerate}

A qualitative interpretation of Theorem~\ref{block:sizes:extension} is thus as follows, all relations of course being approximate: the number of blocks of size~$4$ or more is 0. 
The number of blocks of size~3 is the number of single alignments.
The number of blocks of size~2 is the number of rooks minus twice the number of single alignments.
The number of blocks of size~1 is $n$ minus twice the number of rooks (the factor of two is because each rook combines two singletons), plus the number of single alignments to prevent over double-counting (each alignment converts three singletons into a single block of size 3, whereas a non-alignment converts four singletons into two blocks of size 2).
The qualitative interpretation of Theorem~\ref{cycle:sizes:extension} is the same, with blocks replaced with cycles.

\subsection{Poisson approximation}
\label{poisson:approximation}
Now consider the ${n \choose 2}^r$ ways to place $r$ \emph{distinguishable} rooks on the board $B_n$, 
even allowing two or more rooks on the same square, and consider all
such placements as equally likely.
Write 
$$
   W_{\RR} := \text{ the number of pairs of rooks placed in the same row as each other},
$$
so that $W_\RR$ is a random variable, with $ 0 \le W_\RR \le {r \choose 2}$.
It is easy to see the asymptotic relation, that for each $1 \le a < b \le r$, we have 
$$ 
\p(  \text{rooks $a,b$ are placed in the same row as each other}) \sim \frac{4}{3n}
$$
(see Equation~\eqref{parr}), hence, if $r,n \to \infty$, by the linearity of expectation we have
$$
  \e W_\RR \sim \frac{2 r^2}{3n}.
$$
The same considerations hold for $W_\CC$, the number of pairs of rooks placed in the same column as each other.

For any $t \in (0,\infty)$, for the regime in which $n,k \to \infty$ with
\begin{equation}\label{def r}
   r   \sim t \sqrt{n},
\end{equation}
the net result of the above is that the expected numbers  of attacks have nonzero limits, given by 
\begin{equation}\label{attack limits}
 \e (W_\RR  +W_\CC) \to \frac{4}{3} \, t^2,   \ \   \e W_\CC   \to \frac{2}{3} \, t^2.
\end{equation}

Not surprisingly, Poisson approximations for the situation of \eqref{def r} are valid, implying that
$$
    \p(W_\RR+W_\CC = 0) \to \exp \left( -\frac{4}{3} \, t^2 \right), 
      \ \  \p(W_\CC = 0) \to \exp \left( -\frac{2}{3} \, t^2 \right) .
$$
Combining this with  the bijections for set partitions counted by  $S(n,k)$, and permutations counted by $|s(n,k)|$, and the placement of $r$ non-attacking rooks on the board $B_n$, the result is that in the regime given by \eqref{def r},  asymptotics for the Stirling numbers are given by
\begin{align}\label{result 1}
\begin{split}
S(n,k) & \sim   \frac{1}{(n-k)!} {n \choose 2}^{n-k} \exp \left(-\frac{4}{3} \, t^2 \right),   \\
|s(n,k)| & \sim  \frac{1}{(n-k)!}  {n \choose 2}^{n-k} \exp \left(-\frac{2}{3} \, t^2 \right).
\end{split}
\end{align}
Indeed,  
\cite{MoserWyman2, MoserWyman1}  and
\cite{Louchard1, Louchard2}
provide asymptotics for Stirling numbers, but fail to provide quantitative bounds for the regime
in \eqref{def r},  while \cite{Stirling} gives a version of 
\eqref{result 1}, including preasymptotic bounds, by using the Chen--Stein method for Poisson approximation. 

The above considerations only involved the event of having no attacks, when $r$ rooks are placed independently and uniformly distributed over the board $B_n$; that is, the only information extracted from the bijection for non-attacking rooks is equi-numerosity.
However, as was shown in Section~\ref{the:bijection}, the bijection also determines the entire block structure of the set partition, or cycle structure of the permutation.
In what follows, for the sake of uniform terminology, we will describe the cycle structure of a permutation as its \emph{block} structure.  Conditional on there being no attacks, the placement of $r$ independent rooks determines the block structure via the indicators of alignments: for set partitions, these are the CR and RC coincidences, and for permutations, these are the RR, CR, and RC coincidences; see Figure~\ref{block 3}.  
A single alignment causes two blocks of size 2 to merge into a block of size 3.
In general, an $\ell$-fold alignment involves $\ell+1$ rooks, and gives rise to a block of size $\ell+2$.
We call a $2$-fold alignment a \emph{double alignment}. 
For the regime given by \eqref{def r}, the expected number of $\ell$-fold alignments, for $\ell \ge 2$, tends to zero,
so with high probability a uniformly random set partition or permutation has no blocks of size 4 or larger.
Furthermore, in this situation, the number of blocks of size 3 is equal to the number of pairs of rooks in alignment.
The expected numbers of alignments for the two cases are given by\footnote{There is a numerical coincidence which caused us much confusion in early versions:  comparing \eqref{attack limits}  with \eqref{alignment limits},  the constants involved are
$\frac{4}{3}$ and $\frac{2}{3}$ for attacks, but   $\frac{2}{3}$ and $\frac{4}{3}$ for alignments;  in each display,  the situation for set partitions is given first, and the situation for permutations is given second.
}

\begin{equation}\label{alignment limits}
 \e (W_\RC + W_\CR)  \to \frac{2}{3} \, t^2,   \ \  \   \e (W_\RR + W_\RC + W_\CR)  \to \frac{4}{3} \, t^2,
 \end{equation}
see Equation~\eqref{parc}.  Hence, for the regime given by \eqref{def r}, by proving a Poisson approximation for the number of alignments, conditional on the event of having 
no attacks, we are able to analyze the full block structure.

The \emph{conditional} Poisson limit required above follows from a Poisson process limit, and uses the full power of the process version of the Chen--Stein method, as given by \cite[Theorem 2]{ArratiaGoldstein}.

\begin{figure}
\begin{subfigure}[h]{0.3\textwidth}
\begin{ytableau}
\none[{\color{gray} 6}] & \none  &\none  &\none  &\none &\none &\none \\ 
\none[{\color{gray} 5}] & \none  &\none  &\none  &\none &\none &\ \\ 
\none[{\color{gray} 4}] & \none  &\none  &\none  &\none &b  &\ \\ 
\none[{\color{gray} 3}] & \none  &\none  &\none  &\ &\  &\ \\ 
\none[{\color{gray} 2}] & \none  &\none  &\ &a  &\ &\ \\ 
\none[{\color{gray} 1}] & \none  &\  &\ &\  &\ & \ \\
\none & \none[{\color{gray} 1}] & \none[{\color{gray} 2}] & \none[{\color{gray} 3}] & \none[{\color{gray} 4}] & \none[{\color{gray} 5}] & \none[{\color{gray} 6}]
\end{ytableau}
\end{subfigure}~
\begin{subfigure}[h]{0.5\textwidth}
\[\xymatrix{
 \{i\}\ar[dr]_{(i,j)} \\
 \{j\} \ar[r]& \{i,j\} \ar[r] & \{i,j,k\} \\
 & \{k\} \ar[ru]^{(j,k)}
 }
\]
\end{subfigure}
\caption{A block of size~$3$ is formed by two aligned rooks. For the case of set partitions, this is a CR or RC alignment, with one rook in coordinate $(i,j)$ and another rook  in coordinate $(j,k)$, $1\leq i < j < k \leq n$.
The situation illustrated has $i=2,j=4,k=5,n=6$.}
\label{block 3}
\end{figure}

\section{Proofs}
\subsection{Statement of the main lemma}
\label{rook_process}
We define the \emph{rook coincidence process} $\X = (X_\alpha)_{\alpha\in I}$, a dependent process of indicator random variables, where the index set $I = \{ \alpha = (\{a,b\},s)\}$ consists of all unordered pairs of rooks $1 \leq a < b \leq r :=n-k$, and a marking $s \in \{\RR, \CC, \RC, \CR, \SS\}$. 
Note that $|I|=5 {r \choose 2}$. 
Here $s=\RR$ means $X_\alpha$ indicates whether rooks $a$ and $b$ are in the same row (but not the same square), $s=\CC$ means $X_\alpha$ indicates whether rooks $a$ and $b$ are in the same column (but not the same square), $s = \RC$ means $X_\alpha$ indicates whether the column number of rook $a$ is the same as the row number of rook $b$, $s = \CR$ means $X_\a$ indicates whether the row number of rook $a$ is the same as the column number of rook $b$, and $s = \SS$ means that rooks $a$ and $b$ occupy the same square.

\begin{remark}{\rm
For every $\a \in I$, we define an index set, $D_\a$, to consist of all indices $\beta \in I$ which share at least one rook with $\a$.
The collection of random variables of the rook coincidence process $\{X_\a\}_{\a\in I}$ is \emph{dissociated} with respect to the family $\{D_\alpha\}$, where dissociation is defined as the requirement that  $X_\a$ is independent of all $X_\beta$ for $\beta \in I \setminus D_\a$. 
}\end{remark}

The following theorem gives a quantitative bound on the total variation distance between the joint distribution of \emph{dependent} Bernoulli random variables and a joint distribution of \emph{independent} Poisson random variables. 
We quote a version below which is a corollary to \cite[Theorem~2]{ArratiaGoldstein}, applicable to a collection of \emph{dissociated} Bernoulli random variables.

\begin{theorem}[\cite{ArratiaGoldstein}]\label{moments:theorem}
Let $I$ denote some index set.  Let $\X := (X_\a)_{\a \in I}$ denote a joint distribution of dissociated indicator random variables. 
Define $\Y := (Y_\alpha)_{\alpha\in I}$, a joint distribution of \emph{independent} Poisson random variables, where $\e Y_\alpha = \e X_\alpha$ for all $\alpha \in I$.  
For each $\a$, let $D_\a$ denote the dependency neighborhood of $X_\a$. 
Let $p_\a := \e X_\a$, $p_{\a\b} := \e X_{\a}X_{\b}$, $\a,\b \in I$, and
\begin{align*} 
b_1 := \sum_{\a \in I} \sum_{\b\in D_\a} p_\a p_\b, & \qquad b_2 := \sum_{\a \in I} \sum_{\substack{\b \ne \a \\ \b\in D_\a}} p_{\a\b}.
\end{align*}
Then we have 
\[ d_{TV}(\L(\X),\L(\Y)) \leq 4(b_1 + b_2). \]
Let the index set $I$ be partitioned into disjoint, non-empty subsets, say $I_1, I_2, \ldots, I_\nu$, and let 
\begin{align*} W_i := \sum_{\a \in I_i} X_\a, & \qquad S_i := \sum_{\a \in I_i} Y_\a. \end{align*}
Then we have 
\[ d_{TV}(\L(W_1, W_2, \ldots, W_\nu),\ \L(S_1, S_2, \ldots, S_\nu)) \leq 4(b_1 + b_2). \]
\end{theorem}

We now apply Theorem~\ref{moments:theorem} to obtain an explicit and completely effective upper bound on the total variation distance between the rook coincidence process and a corresponding joint distribution of independent Poisson random variables. 

\begin{lemma}\label{theorem:main}
For each $n \geq 2$ and $k \geq 1$, let $\X$ denote the rook coincidence process, and let $\Y$ denote a corresponding joint distribution of independent Poisson random variables with $\e \Y = \e \X$.  
Define $r := n-k$, $N := \binom{n}{2}$, and 
\begin{align}\label{eq:b}
\begin{split}
 d\ :=\  4&\binom{r}{2}(2r-3) \cdot \left( \left[20\cdot \binom{n}{3}^2 N^{-4}\right] + 6 \cdot \left[\binom{n}{3}N^{-2} \cdot \binom{n}{2} N^{-2}\right] + \right. \\
& \left. \left[\binom{n}{2} N^{-2}\right]^2\right) + 4\binom{r}{3} N^{-3}\left[92\binom{n}{4}+36\,\binom{n}{3}\frac{5n-7}{4} + 6 \binom{n}{2}\right].
\end{split}
 \end{align}
Then we have 
\begin{align*}
\nonumber d_{TV}(\mathcal{L}(\mathbb{X}), \mathcal{L}(\mathbb{Y})) &  \leq d = O\left(\frac{(n-k)^3}{n^2}\right). 
\end{align*}
Furthermore, let $S_{\RR}$, $S_{\CC}$, $S_{\RC}$ $S_{\CR}, S_{\SS}$ denote independent Poisson random variables with expected values $\e \WR$, $\e \WC$, $\e W_{\RC}$, $\e W_{\CR}$, and $\e W_{\SS}$, respectively. 
We have
\[ d_{TV}(\L(\WR, \WC, W_{\RC}, W_{\CR}, W_{\SS}), \L(S_{\RR}, S_{\CC}, S_{\RC}, S_{\CR}, S_{\SS})) \leq d = O\left(\frac{(n-k)^3}{n^2}\right). \]
\end{lemma}

We prove Lemma~\ref{theorem:main} in the next section. 

\subsection{Proof of Lemma~\ref{theorem:main}}
\label{proof:of:bounds}
Recall the index set $I$ from Section~\ref{rook_process}, and we further partition set~$I$ into the following four index sets
\begin{align*}
 \IRR & :=  \{(\{a,b\}, \RR), 1 \leq a < b \leq r \} \\
 \ICC & :=  \{(\{a,b\}, \CC), 1 \leq a < b \leq r \}, \\
 \IRC & :=  \{(\{a,b\}, \RC), 1 \leq a < b \leq r \}, \\
 \ICR & :=  \{(\{a,b\}, \CR), 1 \leq a < b \leq r \}, \\
 \ISS & :=  \{(\{a,b\}, \SS), 1 \leq a < b \leq r \},
\end{align*}
where
\[  I = I_{\RR} \cup I_{\CC} \cup I_{\CR} \cup I_{\RC} \cup I_{\SS}.\]
Recall, also from Section~\ref{rook_process}, the definition of the indicator random variables $X_\alpha$, $\alpha\in I$. 

Define 
\begin{align*}
p_\alpha := \e X_\alpha,  & \qquad p_{\alpha\beta} := \e X_{\alpha}X_{\beta}, \qquad \qquad \alpha,\beta\in I;
\end{align*}
\begin{align*}
b_1^A := \sum_{\alpha \in \IRR \cup \ICC} \sum_{\beta} p_\alpha p_\beta, & \qquad b_2^A  := \sum_{\alpha \in \IRR \cup \ICC} \sum_{\alpha \neq \beta} p_{\alpha \beta}, \end{align*}
where the sum in $\beta$ is over those indices $\beta \in \IRR \cup \ICC$ which share a rook with $\alpha$;  
\begin{align*}
b_1^L := \sum_{\alpha \in \IRC\cup \ICR} \sum_{\beta} p_\alpha p_\beta, & \qquad b_2^L  := \sum_{\alpha \in \IRC \cup \ICR} \sum_{\alpha \neq \beta} p_{\alpha \beta}, \end{align*}
where the sum in $\beta$ is over those indices $\beta \in \IRC \cup \ICR$ which share a rook with $\alpha$; 
\begin{align*} 
b_1^{AL} := 2\, \sum_{\alpha \in \IRR \cup \ICC}  \sum_{\beta} p_\alpha p_\beta,  &\qquad  b_2^{AL} := 2\, \sum_{\alpha \in \IRR \cup \ICC} \sum_{\alpha \neq \beta} p_{\alpha \beta}, \end{align*}
where the sum in $\beta$ is over those indices $\beta \in \IRC \cup \ICR$ which share a rook with $\alpha$;
\begin{align*} 
b_1^{SS} := \sum_{\alpha}  \sum_{\beta} p_\alpha p_\beta,  &\qquad  b_2^{SS} := \sum_{\alpha} \sum_{\alpha \neq \beta} p_{\alpha \beta}, \end{align*}
where the sum is over those $\alpha, \beta \in I$ in which $\alpha, \beta$ share a rook, and at least one of $\alpha, \beta \in \ISS$;
and finally
\begin{align*} 
b_1 & := b_1^A + b_1^L + b_1^{AL} + b_1^{SS}, \\
b_2 & := b_2^A + b_2^L + b_2^{AL} + b_2^{SS}, \\
d & := 4(b_1 + b_2). 
\end{align*}

\begin{lemma}\label{b lemma}
Let $N := \binom{n}{2}$ and $r := n-k$.  We have

\begin{align}
\label{eq:lambda} p_\alpha = 2 \binom{n}{3} N^{-2}, 
 & \qquad \mbox{for all } \alpha \in \IRR \cup \ICC,   \\
\label{eq:lambda2} p_\alpha = \binom{n}{3}\, N^{-2}, & \qquad \mbox{for all } \alpha \in \IRC \cup \ICR, \\
\label{eq:ss} p_\alpha = \binom{n}{2} N^{-2}, & \qquad \mbox{for all } \alpha \in \ISS, 
\end{align}
\begin{align}
\label{b1A} b_1^{A}  =  &\, \binom{r}{2}\, (2r-3) \cdot 2\cdot \left[2 \binom{n}{3} N^{-2}\right]^2\, ,\\ 
\label{b1L} b_1^{L}  =  &\, \binom{r}{2}\, (2r-3)\, \cdot 4 \cdot \left[\binom{n}{3} N^{-2}\right]^2\, ,\\ 
\label{b1AL} b_1^{AL}  = &\, \binom{r}{2}\, (2r-3)\, \cdot  4 \cdot \left[2\binom{n}{3}N^{-2}\right] \left[\binom{n}{3} N^{-2}\right] ,\\ 
\label{b1SS} b_1^{SS}  = &\, \binom{r}{2}\, (2r-3)\, \left( 6 \cdot \left[\binom{n}{3}N^{-2} \cdot \binom{n}{2} N^{-2}\right] + \left[\binom{n}{2} N^{-2}\right]^2 \right),\\ 
\label{b1}   b_1  = &\, \binom{r}{2}(2r-3) \cdot \left( \left[20\cdot \binom{n}{3}^2 N^{-4}\right] + 6 \cdot \left[\binom{n}{3}N^{-2} \cdot \binom{n}{2} N^{-2}\right] + \left[\binom{n}{2} N^{-2}\right]^2\right)\, ,
\end{align} 
\begin{align}
\label{b2A}  b_2^A = \binom{r}{3} N^{-3}& \cdot 12\cdot \left[6\cdot \binom{n}{4} + \binom{n}{3}(5n-11) / 4 \right]\, ,\\
\label{b2L}  b_2^L  = \binom{r}{3} N^{-3}& \cdot 20 \binom{n}{4}\, ,\\
\label{b2AL}  b_2^{AL} = \binom{r}{3} N^{-3} & \cdot 24 \binom{n}{3} \frac{5n-11}{4}\, , \\
\label{b2SS} b_2^{SS} = \binom{r}{3} N^{-3} & \cdot 6 \left[6\binom{n}{3} + \binom{n}{2}\right],
\end{align}
\begin{align}
\label{b2} b_2  = \binom{r}{3} N^{-3}&\left[92\binom{n}{4}+36\,\binom{n}{3}\frac{5n-7}{4} + 6 \binom{n}{2}\right]\, .
\end{align}
\end{lemma}
\begin{proof}
We make repeated use of the following identity: for all positive integers $m, n$, we have
\[ \sum_{j=0}^n \binom{j}{m} = \binom{n+1}{m+1}; \]
see, e.g., \cite[Chapter 2, Section 12]{feller1}.

\begin{figure}[h]
\begin{subfigure}[t]{0.25\textwidth}
\ytableausetup{centertableaux, boxsize=1.15em}
\begin{ytableau}
\none[{\color{gray} 7}] & \none  &\none  &\none  &\none &\none &\none &\none  \\ 
\none[{\color{gray} 6}] & \none  &\none  &\none  &\none &\none &\none &\  \\ 
\none[{\color{gray} 5}] & \none  &\none  &\none  &\none &\none &\ &\ \\ 
\none[{\color{gray} 4}] & \none  &\none  &\none  &\none &\  &\  &\ \\ 
\none[{\color{gray} 3}] & \none  &\none  &\none  &\ &\  &\ &\ \\ 
\none[{\color{gray} 2}] & \none  &\none  &\bullet & a  &\bullet &b &\bullet \\ 
\none[{\color{gray} 1}] & \none  &\  &\ &\  &\ &\ &\ \\
\none & \none[{\color{gray} 1}] & \none[{\color{gray} 2}] & \none[{\color{gray} 3}] & \none[{\color{gray} 4}] & \none[{\color{gray} 5}] & \none[{\color{gray} 6}] & \none[{\color{gray} 7}]
\end{ytableau}
\caption{$\IRR \cup \IRR$}\label{block3_RR_RR}
\end{subfigure}~\qquad
\begin{subfigure}[t]{0.25\textwidth}
\begin{ytableau}
\none[{\color{gray} 7}] & \none  &\none  &\none  &\none &\none &\none &\none  \\ 
\none[{\color{gray} 6}] & \none  &\none  &\none  &\none &\none &\none &\   \\ 
\none[{\color{gray} 5}] & \none  &\none  &\none  &\none &\none &\bullet &\  \\ 
\none[{\color{gray} 4}] & \none  &\none  &\none  &\none &\  &\bullet &\  \\ 
\none[{\color{gray} 3}] & \none  &\none  &\none  &\bullet &\  &\bullet &\  \\ 
\none[{\color{gray} 2}] & \none  &\none  &\ & a  &\ &b &\  \\ 
\none[{\color{gray} 1}] & \none  &\  &\ &\bullet  &\ & \bullet &\  \\
\none & \none[{\color{gray} 1}] & \none[{\color{gray} 2}] & \none[{\color{gray} 3}] & \none[{\color{gray} 4}] & \none[{\color{gray} 5}] & \none[{\color{gray} 6}] & \none[{\color{gray} 7}]
\end{ytableau}
\caption{$\IRR \cup \ICC$}\label{block3_RR_CC}
\end{subfigure}~\qquad
\begin{subfigure}[t]{0.25\textwidth}
\begin{ytableau}
\none[{\color{gray} 7}] & \none  &\none  &\none  &\none &\none &\none &\none  \\ 
\none[{\color{gray} 6}] & \none  &\none  &\none  &\none &\none &\none &\bullet  \\ 
\none[{\color{gray} 5}] & \none  &\none  &\none  &\none &\none &\ &\ \\ 
\none[{\color{gray} 4}] & \none  &\none  &\none  &\none &\bullet  &\bullet  &\bullet \\ 
\none[{\color{gray} 3}] & \none  &\none  &\none  &\ &\  &\ &\ \\ 
\none[{\color{gray} 2}] & \none  &\none  &\ & a  &\ &b &\ \\ 
\none[{\color{gray} 1}] & \none  &\  &\ &\  &\ &\ &\ \\
\none & \none[{\color{gray} 1}] & \none[{\color{gray} 2}] & \none[{\color{gray} 3}] & \none[{\color{gray} 4}] & \none[{\color{gray} 5}] & \none[{\color{gray} 6}] & \none[{\color{gray} 7}]
\end{ytableau}
\caption{$\IRR \cup \IRC$}\label{block3_RR_RC}
\end{subfigure}
\\ 
\begin{subfigure}[t]{0.25\textwidth}
\begin{ytableau}
\none[{\color{gray} 7}] & \none  &\none  &\none  &\none &\none &\none &\none  \\ 
\none[{\color{gray} 6}] & \none  &\none  &\none  &\none &\none &\none &\  \\ 
\none[{\color{gray} 5}] & \none  &\none  &\none  &\none &\none &\ &\ \\ 
\none[{\color{gray} 4}] & \none  &\none  &\none  &\none &\  &\  &\ \\ 
\none[{\color{gray} 3}] & \none  &\none  &\none  &\ &\  &\ &\ \\ 
\none[{\color{gray} 2}] & \none  &\none  &\ & a  &\ &b &\ \\ 
\none[{\color{gray} 1}] & \none  &\bullet  &\ &\  &\ &\ &\ \\
\none & \none[{\color{gray} 1}] & \none[{\color{gray} 2}] & \none[{\color{gray} 3}] & \none[{\color{gray} 4}] & \none[{\color{gray} 5}] & \none[{\color{gray} 6}] & \none[{\color{gray} 7}]
\end{ytableau}
\caption{$\IRR \cup \ICR$}\label{block3_RR_CR}
\end{subfigure} 
~\qquad
\begin{subfigure}[t]{0.25\textwidth}
\begin{ytableau}
\none[{\color{gray} 7}] & \none  &\none  &\none  &\none &\none &\none &\none  \\ 
\none[{\color{gray} 6}] & \none  &\none  &\none  &\none &\none &\none &\  \\ 
\none[{\color{gray} 5}] & \none  &\none  &\none  &\none &\none &\bullet &\bullet \\ 
\none[{\color{gray} 4}] & \none  &\none  &\none  &\none &b  &\bullet  &\bullet \\ 
\none[{\color{gray} 3}] & \none  &\none  &\none  &\ &\  &\ &\ \\ 
\none[{\color{gray} 2}] & \none  &\none  &\ & a  &\ &\ &\ \\ 
\none[{\color{gray} 1}] & \none  &\ &\ &\  &\ &\ &\ \\
\none & \none[{\color{gray} 1}] & \none[{\color{gray} 2}] & \none[{\color{gray} 3}] & \none[{\color{gray} 4}] & \none[{\color{gray} 5}] & \none[{\color{gray} 6}] & \none[{\color{gray} 7}]
\end{ytableau}
\caption{$\IRC \cup \IRC$}\label{block3_RC_RC}
\end{subfigure}
~\qquad
\begin{subfigure}[t]{0.25\textwidth}
\begin{ytableau}
\none[{\color{gray} 7}] & \none  &\none  &\none  &\none &\none &\none &\none  \\ 
\none[{\color{gray} 6}] & \none  &\none  &\none  &\none &\none &\none &\  \\ 
\none[{\color{gray} 5}] & \none  &\none  &\none  &\none &\none &\ &\ \\ 
\none[{\color{gray} 4}] & \none  &\none  &\none  &\none &b  &\  &\ \\ 
\none[{\color{gray} 3}] & \none  &\none  &\none  &\bullet &\  &\ &\ \\ 
\none[{\color{gray} 2}] & \none  &\none  &\ & a  &\ &\ &\ \\ 
\none[{\color{gray} 1}] & \none  &\bullet &\ &\bullet  &\ &\ &\ \\
\none & \none[{\color{gray} 1}] & \none[{\color{gray} 2}] & \none[{\color{gray} 3}] & \none[{\color{gray} 4}] & \none[{\color{gray} 5}] & \none[{\color{gray} 6}] & \none[{\color{gray} 7}]
\end{ytableau}
\caption{$\IRC \cup \ICR$}\label{block3_RC_CR}
\end{subfigure}
\caption{Cases for combinations of two indices in $I$, holding rooks $a<b$ fixed, and each $\bullet$ denotes a possible placement of a third rook which satisfies the set in each case, where $b < \bullet$.  
}
\label{b cases}
\end{figure}

Let $p_1$ denote the probability that two given rooks $a$ and $b$ lie in the same row (but not the same square).
We have
\begin{equation}\label{parr} p_1 = N^{-2}\, \sum_{\ell=2}^{n-1} \ell\,(\ell-1) = 2 \binom{n}{3} N^{-2} \,. \end{equation}
By symmetry, $p_1$ is also the probability that two given rooks $a$ and $b$ lie in the same column.
This establishes Equation~\eqref{eq:lambda}. 

Let $p_2$ denote the probability that two given rooks $a$ and $b$, with $a < b$, are such that the column number of rook $a$ is the row number of rook $b$. 
We have
\begin{equation}\label{parc} p_2 =  N^{-2}\, \sum_{\ell=1}^{n-1} \sum_{c=\ell+1}^{n} (n-c) = N^{-2}\, \binom{n}{3}\, . \end{equation}
By symmetry, $p_2$ is also the probability that two given rooks $a$ and $b$, with $a<b$, are such that the row number of rook $a$ is the column number of rook $b$.
This establishes Equation~\eqref{eq:lambda2}. 

Let $p_3$ denote the probability that two given rooks $a$ and $b$, with $a < b$, are such that the two rooks lie in the same square.  We have
\begin{equation}\label{pass} p_3 =  \binom{n}{2}\, N^{-2}\, ,
\end{equation}
which establishes Equation~\eqref{eq:ss}.

Consider next the term $b_1^A$, which is the total contribution of terms involving two pairs of attacking rooks.
The outer sum for $b_1^A$ is over $2\binom{r}{2}$ pairs of rooks, one factor of $\binom{r}{2}$ for row-attacking and another factor of $\binom{r}{2}$ for column-attacking, and the inner sum is over all indices $\beta \in \IRR \cup \ICC$ which share \emph{at least} one rook with $\alpha$; that is, $2(r-2)+1=2r-3$ cases for overlapping rook(s).
We have
\[ b_1^{A} = \binom{r}{2} (2r-3) 4p_1^2, \]
where the factor of 4 comes from the cases $(\alpha, \beta) \in (\IRR, \IRR), (\ICC, \ICC), (\IRR, \ICC), (\ICC, \IRR)$, which establishes Equation~\eqref{b1A}.  
Similarly, we have
\[ b_1^{L} = \binom{r}{2} (2r-3) 4p_2^2, \]
where the factor of 4 comes from the cases $(\alpha, \beta) \in (\IRC, \IRC), (\IRC, \ICR), (\ICR, \IRC), (\ICR, \ICR)$, which establishes Equation~\eqref{b1L}.
Also, we have
\[ b_1^{AL} = \binom{r}{2}(2r-3) 2\left( p_1(2p_2) + p_2(2p_1) \right), \]
which covers the cases $(\alpha, \beta) \in (\IRR, \IRC), (\IRR, \ICR), (\ICC, \IRC), (\ICC, \ICR)$ and $(\beta, \alpha) \in (\IRR, \IRC), (\IRR, \ICR), (\ICC, \IRC), (\ICC, \ICR)$, which establishes Equation~\eqref{b1AL}.  
Finally,
\[ b_1^{SS} = \binom{r}{2}(2r-3) \left( 2p_1p_3 + 2p_2p_3 + p_3^2 \right), \]
which covers the cases $(\alpha, \beta) \in (\IRR, \ISS), (\ICC, \ISS), (\IRC, \ISS), (\ICR, \ISS), (\ISS, \ISS)$ and $(\beta, \alpha) \in (\IRR, \ISS), (\ICC, \ISS), (\IRC, \ISS), (\ICR, \ISS),$ noting that $(\ISS, \ISS)$ only appears once,
which establishes Equation~\eqref{b1SS}.
Adding these three expressions together gives the expression in Equation~\eqref{b1}.

The expressions involving $b_2$ are more complicated.
We consider first
\[ \small b_2^A = \sum_{\a \in \IRR} \left( \sum_{\substack{\a \ne \b \\ \b\in D_\a \cap \IRR}} p_{\a\b} + \sum_{\substack{\a \ne \b \\ \b\in D_\a \cap \ICC}} p_{\a\b}\right) + \sum_{\a \in \ICC} \left( \sum_{\substack{\a \ne \b \\ \b\in D_\a \cap \IRR}} p_{\a\b} + \sum_{\substack{\a \ne \b \\ \b\in D_\a \cap \ICC}} p_{\a\b}\right), \]
where $D_\a \subset I$ refers to the set of all indices $\b$ for which $\a$ and $\b$ share at least one rook. 
By symmetry, the two outer summations are the same, and so we consider only the first sum since this implies that also
\[ b_2^A = 2\sum_{\a \in \IRR} \left( \sum_{\substack{\a \ne \b \\ \b\in D_\a \cap \IRR}} p_{\a\b} + \sum_{\substack{\a \ne \b \\ \b\in D_\a \cap \ICC}} p_{\a\b}\right). \]

We first consider $\alpha, \beta\in I_{RR}$, where $\beta$ shares exactly one rook with $\alpha$, as depicted in Figure~\ref{block3_RR_RR}. 
We have 
\begin{equation} \label{compute:RR_RR}
p_{\alpha\beta} =  N^{-3}\cdot \sum_{\ell=1}^{n-3} 6\, \binom{n-\ell}{3} = N^{-3}\cdot 6\cdot \sum_{\ell=3}^{n-1} \binom{\ell}{3} = N^{-3}\cdot 6\cdot \binom{n}{4}.
\end{equation}
The summation is now over all $\binom{r}{2}\cdot (2(r-2)) = (r)_3$ ways of selecting an unorderd pair of distinct rooks, followed by selecting another unordered pair of distinct rooks which share exactly one rook.

Next, the probability that two rooks share a row and a third rook shares one of their columns, i.e., $\alpha \in \IRR$ and $\beta \in \ICC$, or vice versa, as depicted in Figure~\ref{block3_RR_CC}, is given by 
\[ p_{\alpha\beta} = 2\cdot N^{-3} \sum_{\ell=1}^{n-2} \sum_{c_1=\ell+1}^{n-1}\sum_{c_2=c_1+1}^n\left[(c_1-2) + (c_2-2)\right] = 2\cdot N^{-3}\left[\binom{n}{3}\frac{(5n-11)}{4}\right]. \]
The outer factor of 2 is due to swapping the locations of the two rooks which share the same row.
The summation is now over all $\binom{r}{2} (r-2) = (r)_3/2$ ways of selecting an unordered pair of distinct rooks, followed by selecting another distinct rook which will share either the same column as the first rook (via the factor $c_1-2$) or the second rook (via the factor $c_2-2$).

Summing over these two cases, we have
\begin{equation}\label{compute:b2A}
 b_2^A = (r)_3 \cdot N^{-3} \cdot 2\cdot \left[6\cdot \binom{n}{4} + \binom{n}{3}\frac{5n-11}{4} \right],
 \end{equation}
where the outer factor of 2 is from exchanging the roles of $\RR$ and $\CC$.  This establishes Equation~\eqref{b2A}. 

Next, for alignments, there are two main cases, as depicted in Figure~\ref{block3_RC_RC} and Figure~\ref{block3_RC_CR}, which rely on the assumption that $a < b < \bullet$.
We start with $\alpha \in \IRC$ and $\beta \in \IRC$, which corresponds to Figure~\ref{block3_RC_RC}. 
Let us consider first a 2-fold (i.e., double) alignment; that is, three rooks, say $a<b<c$, where the row number of $a$ is equal to the column number of $b$, and the row number of $b$ is equal to the column number of $c$. 
In this case, none of the rooks are attacking, and all lie on distinct squares. 
We have 
\begin{equation*}
p_{\alpha\beta} = N^{-3}\sum_{\ell=1}^{n-3}\sum_{c_1=\ell+1}^{n-2} \sum_{c_2=c_1+1}^{n-1} (n-c_2) = N^{-3} \binom{n}{4}\ . \end{equation*}
The number of such unordered triplets of rooks which satisfy $1 \leq a<b<c \leq r$ is $\binom{r}{3}$.

The next case is when there are two alignments (but \emph{not} a double alignment) due to rooks $b$ and $c$ being in the same row, both aligned with rook $a$, with $a < b$ and $a < c$.
We have
\[ p_{\alpha\beta} = N^{-3} \sum_{\ell=1}^{n-3} \sum_{c_1=\ell+1}^{n-2} (n-c_1)\, (n-c_1-1) = N^{-3}\, 2\, \binom{n}{4}. 
\]
The number of such triplets of distinct rooks which satisfy $a < b$ and $a < c$ is given by
\[ \sum_{r_1 = 1}^{r-1} \sum_{r_2=r_1+1}^{r} (r-r_1-1) = 2\, \binom{r}{3}. \]
Combining these we have 
\begin{equation*}
b_{\RC,\,\RC} := \sum_{\alpha \in \IRC} \sum_{\substack{\a \ne \b \\ \beta \in D_\alpha, \beta \in \IRC}} p_{\alpha\beta} = \binom{n}{4} \binom{r}{3} N^{-3} + 2\binom{n}{4}2\,\binom{r}{3} N^{-3} = 5\binom{n}{4}\binom{r}{3} N^{-3}. 
\end{equation*}

We next consider $\alpha \in \IRC$ and $\beta \in \ICR$, which corresponds to Figure~\ref{block3_RC_CR}.  A similar calculation (or argument by symmetry) yields 
\begin{equation*}
b_{\RC,\,\CR} := \sum_{\alpha \in \IRC} \sum_{\beta \in D_\alpha, \beta \in \ICR} p_{\alpha\beta} = 5\, \binom{n}{4}\binom{r}{3} N^{-3}. 
\end{equation*}
All remaining cases for the summands in $b_2^L$ follow by symmetry, as substituting $\RC$ with $\CR$ and vice versa in both summands does not change the value.
It follows that
\begin{equation}\label{compute:b2L}
 b_2^L  = \sum_{\alpha \in \IRC \cup \ICR} \sum_{\alpha\ne \beta} p_{\alpha\beta} = 2\, b_{\RC,\, \RC} + 2\, b_{\RC, \CR} = 20 \binom{r}{3} \binom{n}{4} N^{-3}. 
\end{equation}

The cases for both an attack and alignment are depicted in Figure~\ref{block3_RR_RC} and Figure~\ref{block3_RR_CR}.
For $\alpha \in \IRR, \beta \in \IRC$, we consider the case where $a < b < c$ and $\beta = \{a, c\} \in \IRC$, then we have
\[ p_{\alpha\beta} = \sum_{\ell=1}^{n-2} \sum_{c_1=\ell+1}^{n} (n-\ell-1) (n-c_1) = \binom{n}{3} \frac{3n-5}{4}, \]
with the summations extending over $\binom{r}{3}$ terms.

Continuing with $\alpha \in \IRR, \beta \in \IRC$, in the case $\beta = \{b, c\} \in \IRC$, then we have the same value of $p_{\alpha\beta}$ but with a summation extending over $2\binom{r}{3}$ terms.
In the case where $\beta = \{a,c\}$ and $c < a$, this is the same as having two alignments that is not a double alignment, and so $p_{\alpha\beta} = 2\binom{n}{4}$, with $\binom{r}{3}$ summands.
In the case where $\beta = \{b, c\}$ and $c < b$, $p_{\alpha\beta} = 2\binom{n}{4}$ similarly, with $2\binom{r}{3}$ summands. 

Summing over these cases, we have 
\begin{align}
\nonumber b_2^{AL} & =  2 \sum_{\alpha \in \IRR \cup \ICC} \sum_{\substack{\a\ne\b \\ \beta \in D_\alpha \cap (\IRC \cup \ICR)}}  p_{\alpha\beta} = 4 \sum_{\alpha \in \IRR \cup \ICC} \sum_{\substack{\b\ne\a \\ \beta \in D_\alpha \cap \IRC}}p_{\alpha\beta} \\ 
\nonumber	       & = 8 \sum_{\alpha \in \IRR} \sum_{\substack{\a\ne\b \\ \beta \in D_\alpha \cap \IRC}}  p_{\alpha\beta}  = 8 \times \left[3 \binom{r}{3}\, \binom{n}{3} \frac{3n-5}{4}\, N^{-3} + 3\binom{r}{3} 2\binom{n}{4} N^{-3}\right] \\
\label{compute:b2AL} 	       & = 24\binom{r}{3} \binom{n}{3} \frac{5n-11}{4}\, N^{-3}. 
\end{align}

Finally, the summands in $b_2^{SS}$ are easily seen to be equivalent to the previous cases involving one pair of rooks, as in
\begin{align}
\nonumber b_2^{SS} & = 2\,(r-2) \left[ \sum_{\alpha \in \IRR} p_\alpha / N + \sum_{\alpha \in \ICC} p_\alpha / N + \sum_{\alpha \in \IRC} p_\alpha / N + \sum_{\alpha \in \ICR} p_\alpha / N + \sum_{\alpha\in\ISS} N / N^{3}\right] \\
\nonumber  & = (r)_3 N^{-3} \left[2\binom{n}{3} + 2\binom{n}{3} + \binom{n}{3} + \binom{n}{3} + \binom{n}{2}\right] \\
\label{compute:b2SS}   & = (r)_3 N^{-3} \left[6\binom{n}{3} + \binom{n}{2}\right].
\end{align}

Adding equations~\eqref{compute:b2A},~\eqref{compute:b2L},~\eqref{compute:b2AL},~and~\eqref{compute:b2SS}, we obtain the expression for $b_2,$ which establishes Equation~\eqref{b2}. 
\end{proof}

Thus, Lemma~\ref{theorem:main} follows from equations~\eqref{b1}~and~\eqref{b2} via $d = 4(b_1+b_2)$, which is an upper bound on the total variation distance by Theorem~\ref{moments:theorem}.

\subsection{Proof of Theorems~\ref{cycle:sizes:extension} and~\ref{block:sizes:extension}}
To prove our main theorems we must condition on the event that there are no pairs of attacking rooks. 

We start with the joint distribution governing random set partitions. 
Define 
\[ W_2 := W_{\RC} + W_{\CR}, \]
which is the sum of indicators for pairwise occurrence of an alignment; and define
\[ R_2 := \sum_{\alpha \in \ICR\cup \ICR} \sum_{\alpha \neq \beta} X_\alpha X_\beta, \]
where the sum in~$\beta$ is over those indices $\beta \in \IRC \cup \ICR$ which share at least one rook with $\alpha$. 
Since $R_2$ is defined as the sum of indicators, the event $\{R_2 = 0\}$ also implies that there are no triple alignments, etc.
We therefore have
\begin{align*}
 d_{TV}( & \bfD,  (n-2r+W_2, r-2 W_2, W_2, 0, \ldots)\, |\,  \WC + \WR +\WS = 0) \\
 & = d_{TV}( \L(R_2\, |\, \WR+ \WC + \WS = 0), 0 ), 
\end{align*}
where the right hand side is simply $\p(R_2>0 | \WR+ \WC +\WS = 0)$.  
We have
\[ \p(R_2>0 | \WR+ \WC +\WS = 0) = \frac{\p(R_2>0, \WR + \WC +\WS = 0)}{\p(\WR + \WC +\WS = 0)} \leq \frac{\e R_2}{e^{-\lambda_R -\lambda_C-\lambda_{SS}} - d}, \]
where $d$ is given in Equation~\eqref{eq:b}, 
\[ \lambda_R = \lambda_C = \binom{r}{2} p_\alpha = \frac{2\binom{r}{2} \binom{n}{3} }{ \binom{n}{2}^2} , \]
where $p_\alpha$ is given in Equation~\eqref{eq:lambda}, and $\lambda_{SS} = \binom{r}{2}/\binom{n}{2}$, as long as $e^{-\lambda_R -\lambda_C - \lambda_{SS}} - d > 0$.  
By combining these expressions with Equation~\eqref{b2L} for $\e R_2$, we have 
\[ \p(R_2>0 | \WR+ \WC +\WS = 0) \leq \frac{20 \binom{r}{3}\binom{n}{4} / \binom{n}{2}^3}{\exp\left( -\lambda_R-\lambda_C-\lambda_{SS}\right) - d}, \]
with
\[ \frac{20 \binom{r}{3}\binom{n}{4} / \binom{n}{2}^3}{\exp\left( -\lambda_R-\lambda_C-\lambda_{SS}\right) - d} = O\left(\exp\left(\frac{4}{3}\, \frac{r^2}{n}\right) \frac{r^3}{n^2}\right) \]
whenever $r  = O(\sqrt{n})$. 
This completes the proof of Theorem~\ref{block:sizes:extension}.

For random permutations, we define similarly
\[ W_1 := W_{\RC} + W_{\CR} + W_{\RR} \]
and 
\[ R_1 := \sum_{\alpha\in \IRR} \sum_{\alpha \neq \beta} X_\alpha X_\beta + \sum_{\alpha\in\IRC\cup\ICR} \sum_{\alpha \neq \beta} X_\alpha X_\beta, \]
where in the first term the sum in~$\beta$ is over those indices $\beta \in \IRR \cup \IRC \cup \ICR$ which share at least one rook with $\alpha$, and in the second term the sum in~$\beta$ is over those indices $\beta \in \IRC\cup \ICR$ which share at least one rook with $\alpha$; i.e., $R_1$ is the sum of all indicator random variables of the event that three rooks form a double alignment in the corresponding bijection for permutations. 
We have similarly,
\[ d_{TV}( \bfC, (n-2r+W_1, r-2W_1, W_1, 0, \ldots)\,|\, \WC + \WS = 0) = d_{TV}( \L(R_1\, |\, \WC+\WS = 0), 0 ),  \]
where 
\[ \p(R_1>0\, |\, \WC + \WS = 0) \leq \frac{d/2}{\exp\left( -\lambda_C-\lambda_{SS}\right) - d}, \]
and where similarly as before, the inequality holds as long as $e^{-\lambda_C-\lambda_{SS}} - d > 0$, with
\[ \frac{d/2}{\exp\left( -\lambda_C-\lambda_{SS}\right) - d} = O\left(\exp\left(\frac{2}{3}\frac{r^2}{n}\right) \frac{r^3}{n^2}\right) \]
whenever $r = O(\sqrt{n})$. 

This completes the proof of Theorem~\ref{cycle:sizes:extension}. 

\section{Applications}
\label{applications}

The total variation distance bounds in Lemma~\ref{theorem:main} can also be used to obtain inequalities for the Stirling numbers. 
See also~\cite{Stirling} for similar bounds, and~\cite{PIPARCS} for similar bounds in a more general setting.
We have 
\[ \p(\WR + \WC + \WS = 0) = \frac{(n-k)!\, S(n,k)}{\binom{n}{2}^{n-k}}, \]
\[ \p(\WC + \WS = 0) = \frac{(n-k)!\, |s(n,k)|}{\binom{n}{2}^{n-k}}. \]
From Lemma~\ref{theorem:main} it follows that for $n-k = O(\sqrt{n})$, $\WC$, $\WR$, and $\WS$ are approximately Poisson distributed, with quantitative bounds provided for all finite values of parameters. 
We prove in Theorem~\ref{Stirling:one:approx} completely effective bounds on Stirling numbers using a sharper inequality especially optimized for the point probability at $0$, presented in Theorem~\ref{sharper} below.
\begin{theorem}[{\cite{ArratiaGoldstein}}]\label{sharper}
Under the assumptions and notation of Theorem~\ref{moments:theorem}, with $W := \sum_{\alpha \in I} X_\alpha$ and $\lambda := \e W$, we have 
\[ \left|\p(W = 0) - e^{-\lambda}\right| \leq \min(1, \lambda^{-1})(b_1 + b_2). \]
\end{theorem}

Note that for this application of Poisson approximation, we use the smaller index set $\IRR \cup \ICC \cup \ISS$ for Stirling numbers of the second kind, and smaller still index set $\ICC \cup \ISS$ for Stirling numbers of the first kind.

\begin{theorem}\label{Stirling:one:approx}
Let $N = \binom{n}{2},$
\[\lambda_1(n,r) := \frac{2\binom{r}{2} \binom{n}{3} }{ \binom{n}{2}^2} + \frac{\binom{r}{2}}{\binom{n}{2}}, \qquad \lambda_2(n,r) := \frac{4\binom{r}{2} \binom{n}{3} }{ \binom{n}{2}^2} + \frac{\binom{r}{2}}{\binom{n}{2}}, \qquad \] 
\[f_1(n,r) :=  \binom{r}{2}(2r-3)\, N^{-4}\, \left(4 \binom{n}{3}^2 + 4 \binom{n}{3}\binom{n}{2} + \binom{n}{2}^2\right), \]
\[ f_2(n,r) :=  (r)_3\, N^{-3} \left(6\binom{n}{4}  + 4 \binom{n}{3} + \binom{n}{2}\right), \]
\[ h_1(n,r) := \binom{r}{2}(2r-3)N^{-4}\left(16 \binom{n}{3}^2 + 8 \binom{n}{3}\binom{n}{2} + \binom{n}{2}^2\right)\]
\[ h_2(n,r) := (r)_3 N^{-3} \left[12\binom{n}{4} + \binom{n}{3}\frac{5n-3}{2} + \binom{n}{2}\right]\, ,\] 
Define $\lambda \equiv \lambda(n, n-k)$, $f_1 \equiv f_1(n,n-k)$, $f_2 \equiv f_2(n, n-k)$, $h_1 \equiv h_1(n, n-k)$, $h_2 \equiv h_2(n, n-k)$. 

Then for each $n \geq 2$ and $1 \leq k \leq n$, we have 
\[ \frac{\binom{n}{2}^{n-k}}{(n-k)!}e^{-\lambda_1}\left(1 - e^{\lambda_1}(f_1+f_2)\right) \leq |s(n,k)| \leq \frac{\binom{n}{2}^{n-k}}{(n-k)!}e^{-\lambda_1}\left(1 + e^{\lambda_1}(f_1+f_2)\right); \]
\[ \frac{\binom{n}{2}^{n-k}}{(n-k)!}e^{-\lambda_2}\left(1 - e^{\lambda_2}(h_1+h_2)\right) \leq S(n,k) \leq \frac{\binom{n}{2}^{n-k}}{(n-k)!}e^{-\lambda_2}\left(1 + e^{\lambda_2}(h_1+h_2)\right). \]
\end{theorem}
\begin{proof}
For Stirling numbers of the first kind, we use index set $\ICC \cup \ISS.$  Recalling the notation from the proof of Lemma~\ref{theorem:main}, we have 
\begin{align*}
 b_1 \, = & \sum_{\alpha \in \ICC \cup \ISS} \sum_{\beta \in D_\alpha \cap (\ICC \cup \ISS)} p_\alpha p_\beta \\
         = & \sum_{\alpha \in \ICC} \sum_{\beta \in D_\alpha \cap \ICC} p_\alpha p_\beta + \sum_{\alpha \in \ICC} \sum_{\beta \in D_\alpha \cap \ISS} p_\alpha p_\beta \\
        & + \sum_{\alpha \in \ISS} \sum_{\beta \in D_\alpha \cap \ICC} p_\alpha p_\beta + \sum_{\alpha \in \ISS} \sum_{\beta \in D_\alpha \cap \ISS} p_\alpha p_\beta \\
        & = \binom{r}{2}(2r-3)\left(p_1^2 + 2p_1p_3 + p_3^2\right) \\
        & = \binom{r}{2}(2r-3)\left(4 \binom{n}{3}^2 N^{-4} + 4 \binom{n}{3}\binom{n}{2} N^{-4} + \binom{n}{2}^2 N^{-4}\right).
\end{align*}
For $b_2$ we have similarly (see Equation~\eqref{compute:RR_RR} and Equation~\eqref{compute:b2SS})
\begin{align*}
b_2 \, = & \sum_{\alpha \in \ICC \cup \ISS} \sum_{\substack{\b\ne\a \\ \beta \in D_\alpha \cap (\ICC \cup \ISS)}}p_{\alpha\beta} \\
& = \sum_{\alpha \in \ICC} \sum_{\substack{\b \ne \a \\ \beta \in D_\alpha \cap \ICC}} p_{\alpha\beta} + 2\sum_{\alpha \in \ICC} \sum_{\substack{\b \ne \a \\ \beta \in D_\alpha \cap \ISS}} p_{\alpha\beta} + \sum_{\alpha \in \ISS} \sum_{\substack{\b \ne \a \\ \beta \in D_\alpha \cap \ISS}} p_{\alpha\beta} \\
& = (r)_3 \left(6\cdot \binom{n}{4}N^{-3}  + 2\cdot \left(2 \binom{n}{3} N^{-3}\right) + \binom{n}{2} N^{-3}\right).
\end{align*}

For Stirling numbers of the second kind, we use index set $\IRR \cup \ICC \cup \ISS$.  A similar calculation as the above yields
\[ b_1 = 4p_1^2 + 4 p_1 p_3 + p_3^2, \]
and
\begin{align*}
 b_2 & = b_2^A + 2\,(r-2) \left[ \sum_{\alpha \in \IRR} p_\alpha / N + \sum_{\alpha \in \ICC} p_\alpha / N + \sum_{\alpha\in\ISS} N / N^{3}\right] \\
 & = (r)_3 \cdot N^{-3} \cdot 2\cdot \left[6\cdot \binom{n}{4} + \binom{n}{3}\frac{5n-11}{4} \right] + (r)_3 N^{-3} \left[2\binom{n}{3} + 2\binom{n}{3} + \binom{n}{2}\right].
 \end{align*}
 Applying Theorem~\ref{sharper} and rearranging the terms completes the proof.
\end{proof}

Another approach for obtaining preasymptotic lower and upper bounds for the sum of dissociated indicator random variables is by using the Lov\'asz local lemma~\cite{erdos1975problems} for the lower bound (see also~\cite{peres2010two}), and Suen's inequality~\cite{suen1990correlation} for the upper bound (see also~\cite{janson2002concentration}).  See for example~\cite{crane2016probability, perarnau} for applications involving pattern-avoidance in random permutations.  
First we state the theorems in terms of dependency graphs and apply them below. 

\begin{theorem}[Lov\'asz local lemma \cite{peres2010two}]\label{LLL}
Let $\{E_i\}_{i=1}^m$ be a collection of events in some probability space, and let $\{x_i\}_{i=1}^m$ be a sequence of numbers in $(0,1)$.  
Let $H$ denote the dependency graph for $\{E_i\}_{i=1}^m$, which is a graph with vertex set $\{1,\ldots,m\}$ such that for disjoint subsets $A$ and $B$ of $\{1,\ldots,m\}$, no edges in $H$ implies that $\{E_i\}_{i \in A}$ and $\{E_i\}_{i \in B}$ are independent. 
Suppose for each $\ell=1,\ldots,m$ we have for real-valued $x_i \in (0,1)$ that 
\begin{equation}\label{lower:condition} 
\p\left(E_\ell\right) \leq x_\ell  \prod_{i\sim j}(1-x_i). 
\end{equation}
Then we have 
\begin{equation}\label{lower:bound} \p\left(\bigcap_{i = 1}^m A_i^c\right) \geq \prod_{i = 1}^m (1-x_i). \end{equation}
\end{theorem}

\begin{theorem}[Suen's inequality~{\cite[Theorem~2]{janson2002concentration}}]\label{Suen_inequality}
Let $\{I_i\}_{i \in\mathcal{I}}$ denote a finite family of indicator random variables defined on a common probability space. 
Let $H$ denote the dependency graph for $\{I_i\}_{i\in \mathcal{I}}$, which is a graph with vertex set $\mathcal{I}$ such that for disjoint subsets $A$ and $B$ of $\mathcal{I}$, no edges in $H$ implies that $\{I_i\}_{i \in A}$ and $\{I_i\}_{i \in B}$ are independent. 
Define random variable $N := \sum_{i=1}^{m} I_i$, and let $\mu_i := \e\, I_i$, $i=1,\ldots, m$. 
Finally, define 
\begin{align*}
 \Delta = \sum_{\{i,j\}:i \sim j} \e I_i I_j, &\qquad \qquad  \delta = \max_i \sum_{j \sim i} \e I_j.
\end{align*}
Then 
\[ \p(N = 0) \leq \exp\left( -\sum_{i\in \mathcal{I}}\mu_i + \Delta  e^{2\delta}\right). \]
\end{theorem}

\begin{theorem}\label{LS2}
Let $N = \binom{n}{2}$.  Let $p_2 \equiv p_2(n) = \frac{4\binom{n}{3}}{N^2} + \frac{1}{N}$ be the probability that two rooks are attacking or lie on the same square, 
and let $p_1 \equiv p_1(n) = \frac{2\binom{n}{3}}{N^2} + \frac{1}{N}$ be the probability that two rooks lie in the same column, possibly the same square, as computed in Equation~\eqref{parr}.
We define
\[ \lambda_1 \equiv \lambda_1(n,r) := p_1(n) \binom{r}{2},\qquad \qquad  \lambda_2 \equiv \lambda_2(n,r) := p_2(n) \binom{r}{2}, \]
\[ g_1(n,r) := \frac{1}{2} \left(1 - p_1 - \sqrt{1 - \left(4\, (2r-2) - 2\right) p_1 + p_1^2}\right), \]
\[ g_2(n,r) := \frac{1}{2} \left(1 - p_2 - \sqrt{1 - \left(4\, (2r-2) - 2\right) p_2 + p_2^2}\right), \]
$\lambda_1 \equiv \lambda_1(n, n-k)$, $\lambda_2 \equiv \lambda_2(n, n-k)$, $g_1 \equiv g_1(n,n-k)$, $g_2 \equiv g_2(n, n-k)$, and recall $f_2$ and $h_2$ from Theorem~\ref{Stirling:one:approx}.

For all $n \geq 2$ and $1 \leq k \leq n$ such that $p_1\, e^{g_1} \in (0,1)$, we have 
\[ \frac{\binom{n}{2}^{n-k}}{(n-k)!} \left(1 - p_1\, e^{g_1}\right)^{\binom{n-k}{2}}  \leq |s(n,k)| \leq \frac{\binom{n}{2}^{n-k}}{(n-k)!}e^{-\lambda_1} \exp\left(f_2  \exp(2\, p_1\, (2(n-k)-3) \right). \]
Also, for all $k$ and $n$ such that $p_2\,  e^{g_2(n,n-k)} \in (0,1)$, we have 
\[ \frac{\binom{n}{2}^{n-k}}{(n-k)!} \left(1 - p_2\, e^{g_2}\right)^{\binom{n-k}{2}}  \leq S(n,k) \leq \frac{\binom{n}{2}^{n-k}}{(n-k)!}e^{-\lambda_2} \exp\left(h_2 \exp(2\, p_2\, (2(n-k)-3) \right). \]
\end{theorem}
\begin{proof}

First we prove the upper bounds using Suen's inequality. 
Recall from Section~\ref{rook_process} the random variables $X_\alpha$, for $\alpha \in I$, where $X_\alpha$ is the indicator of some event, which we now denote by $E_\alpha$. 
Using the notation from Section~\ref{proof:of:bounds}, note that event $E_\alpha$ depends on whether $\alpha$ is in $\IRR$, $\ICC$, $\ISS$.  
We write $E_\alpha^c$ to denote the complementary event. 
Define $\mathcal{E}_1 := \{E_\alpha\}_{\alpha \in \ICC \cup \ISS}$, and $\mathcal{E}_2 := \{E_\alpha\}_{\alpha\in\IRR \cup \ICC \cup \ISS}$. 
We next define the dependency graph $H_1$ of $\mathcal{E}_1$ to be the graph with vertex set $V(H_1) = \ICC \cup \ISS$, and edge set $E(H_1)$, where edges occur in $H_1$ if the nodes corresponding to $\alpha, \beta \in \ICC \cup \ISS$ share a rook. 
We similarly define the dependency graph $H_2$ of $\mathcal{E}_2$ to be the graph with vertex set $V(H_2) = \ICC \cup \IRR \cup \ISS$, and edge set $E(H_2)$, where edges occur in $H_2$ if the nodes corresponding to $\alpha, \beta \in \ICC \cup \IRR \cup \ISS$ share a rook. 

For $j=1,2$, we define 
\[ \Delta_j := \sum_{\{\alpha,\beta\}\in E(H_j)} \p(E_\alpha \cap E_\beta), \] 
and 
\[ \delta_j := \max_{\alpha \in V(H_j)} \sum_{\beta : \{\alpha,\beta\} \in E(H_j)} \p(E_\beta). \]
It is easy to see that $\Delta_1 = \sum_{\alpha \in \ICC \cup \ISS} \sum_{\a\ne\beta \in D_\alpha \cap (\ICC \cup \ISS)} p_{\alpha\beta}$, and hence is equal to $f_2$ computed in the proof of Theorem~\ref{Stirling:one:approx}; similarly, $\Delta_2$ is equal to $h_2$.  
Finally, $\delta_j$ is the sum over all overlapping sets of rooks of the probability that two rooks will be attacking, which is the same for all rooks and hence the maximum does not impact the value.
We are thus able to apply Theorem~\ref{Suen_inequality} separately to the two cases. 

The lower bounds follow from the approach in the proof of~\cite[Proposition~7.7]{crane2016probability}.
\end{proof}

We performed a numerical comparison on four different methods: 
\begin{enumerate}
\item Permutation coupling ~\cite[Theorem~5]{Stirling};
\item Independence coupling ~\cite[Theorem~6]{Stirling};
\item Theorem~\ref{Stirling:one:approx};
\item Theorem~\ref{LS2}. 
\end{enumerate}
Based on numerical calculations using a few large values of $n$, using $r=1,2,\ldots$, the Lov\'asz local lemma appears to outperform the other methods for accuracy of the lower bound, with the method in Theorem~\ref{Stirling:one:approx} coming in a distant second.  
For the upper bound, in each of the examples investigated, the method in Theorem~\ref{Stirling:one:approx} was more accurate than the other competing methods. 

\section*{Acknowledgements}
The authors gratefully acknowledge several anonymous referees whose comments enhanced the readability of the present paper.

\bibliographystyle{acm}

\end{document}